\documentclass{article}
\title{Obstructing Reducible Surgeries: Slice Genus and Thickness Bounds}
\author{Holt Bodish and Robert DeYeso III}
\date{}
\usepackage{amsmath,amsthm, amssymb,amsfonts}
\usepackage[alphabetic]{amsrefs}
\usepackage{mathtools}
\usepackage{mathrsfs}
\usepackage[margin=1.5in]{geometry}
\theoremstyle{plain}
\newtheorem{theorem}{Theorem}[section]
\newtheorem{lemma}[theorem]{Lemma}
\newtheorem{proposition}[theorem]{Proposition}
\newtheorem{definition}[theorem]{Definition}
\newtheorem{corollary}[theorem]{Corollary}
\newtheorem{conjecture}[theorem]{Conjecture}

\newcommand{\Z}{\mathbb{Z}}

\newcommand{\F}{\mathbb{F}}

\newcommand{\HFhat}{\widehat{HF}}
\newcommand{\rank}{\text{rank}}
\newcommand{\Spinc}{\mathrm{Spin^c}}
\newcommand{\Cone}{\mathrm{Cone}}

\begin{document}
\maketitle
\begin{abstract}
    In this paper, we study reducible surgeries on knots in $S^3$. We develop thickness bounds for L-space knots that admit reducible surgeries, and lower bounds on the slice genus for general knots that admit reducible surgeries. The L-space knot thickness bounds allow us to finish off the verification of the Cabling Conjecture for thin knots, which was mostly worked out in \cite{DeY21b}. We also provide a new upper bound on reducing slopes for fibered, hyperbolic slice knots and on multiple reducing slopes for slice knots. Our techniques involve the $d$-invariants and mapping cone formula from Heegaard Floer homology.
\end{abstract}

\section{Introduction}


Suppose $K$ is a knot in $S^3$. We denote by $S^3_{p/q}(K)$ the result of $\frac{p}{q}$-Dehn surgery on $K$. 
We are interested in Dehn surgeries that produce essential $2$-spheres. A $2$-sphere in a $3$-manifold is essential if it is not the boundary of an embedded $3$-ball $B^3$, and we say that $M$ is reducible if it contains an essential $2$-sphere. The solution of the Property R conjecture in \cite{Gabai87} shows we can assume that $\frac{p}{q}\neq 0$ and the surgery decomposes as a connected sum. All known examples of reducible surgeries on knots in $S^3$ are given by $pq$-surgery on the $(p,q)$-cable of a knot $K$. Letting $C_{p,q}(K)$ denote the $(p,q)$-cable of $K$, we have \begin{equation}\label{red}S^3_{pq}(C_{p,q}(K))\cong L(p,q)\#S^3_{q/p}(K).\end{equation} The Cabling Conjecture of Gonz\'{a}les-Acu\~{n}a and Short asserts that these are the only examples.


\begin{conjecture}[\label{CC}\cite{Acuna}] If $K$ is a knot in $S^3$ which has a reducible surgery, then $K$ is a cabled knot and the reducing slope is given by the cabling annulus.

\end{conjecture}

The cabling conjecure is known to be true for many families of knots: satellite knots \cite{Schar}, alternating knots \cite{Menasco}, torus knots \cite{Moser}, and for knots with symmetries and low bridge number (for a survey of known results and techniques see \cite{boyer}.) In particular, it remains to study reducible surgeries on hyperbolic knots.



Much is known about reducible surgeries on general knots. Observe that in the case of cabled knots, the reducing slope is integral and one of the connected summands is a lens space. Gordon and Luecke \cite{Gor87} show that this is the case for any reducing slope. In particular, any reducing slope $r$ satisfies $1<|r|$ since the lens space summand has to be non-trivial. The integrality of the reducing slope, combined with the theorem of Gordon \cite{Gor96} that the geometric intersection number of any two reducing slopes is $1$, shows that if $K$ has two reducing slopes, then they are consecutive integers. Further work in \cite{Sayari} shows that for hyperbolic knots, the reducing slope $r$ satisfies the restrictive bound \begin{equation}\label{Sayaribound}1<|r|\leq2g(K)-1.\end{equation}
It is also known that in a reducible surgery, no more than one of the summands is an integer homology sphere, and at most two of the summands are lens spaces \cites{Howie, Sanchez}. It is conjectured that three summands never arise from Dehn surgery on a knot in $S^3$ (for more on this, see Corollary 1.6). 


Another bound on the reducing slopes, this time for fibered knots with reducible surgeries with integer homology sphere summands, is proved in \cite{Acuna}*{Proposition 1.4}. We state it here for convenience, rephrased from the original source and using the Poincar\'{e} Conjecture. 

\begin{theorem}[\cite{Acuna}*{Proposition 1.4} \label{fiberedHprime}]

Suppose $K$ is a fibered knot in $S^3$ of genus $g$. If $S^3_r(K) \cong L(r,a) \# Y$ for $Y$ a homology sphere, then $r \leq g$.

\end{theorem}

More recently, progress has been made on the Cabling Conjecture using tools from Heegaard Floer homology. In \cite{HLZ15} they show that L-space knots only have one possible reducing slope $r=2g(K)-1$. In \cite{DeY21b} it is shown that hyperbolic thin knots do not admit reducible surgeries, except for the case of hyperbolic thin L-space knots (for more on this, see Theorem \ref{thm:thinCC}). In \cite{Greene} it is shown that the Cabling Conjecture is true for knots that have surgeries to connected sums of lens spaces. In \cite{threesummands} it is shown that reducible surgeries on slice knots (or more generally knots with $V_i(K)=0$ for all $i\geq 0$) only have two summands. 

In the following, we use Heegaard Floer homology to show that L-space knots that admit reducible surgeries have thickness greater than or equal to $2$, useful for finishing the verification of the Cabling Conjecture for thin knots. We also show that slice knots only admit reducible surgeries of a particular type, and more generally we can bound the slice genus of a knot in terms of the reducing slope parameters. The form of the reducible surgery on a slice knot allows us to restrict the possible slopes on fibered, hyperbolic slice knots. Our techniques mostly involve studying differences of the $d$-invariants of surgery, inspired by the proof of \cite{DeY21b}*{Lemma 5.5}.

\begin{theorem}{\label{thm:redslice}}
Suppose $K \subset S^3$ is a hyperbolic slice knot and $p,q$ are relatively prime integers. If $pq$ is a reducing slope for $K$ and $S^3_{pq}(K) \cong L(p,a) \# Y$ with $|H^2(Y;\Z)|=q$, then $q=1$, $a=1$, and $d(Y)=0$.
\end{theorem}





For fibered, hyperbolic slice knots, Theorem \ref{thm:redslice} together with Theorem \ref{fiberedHprime} implies the following, which when compared with Equation \ref{Sayaribound} shows that we cut down the possible reducing slopes on fibered, hyperbolic slice knots by half. 

\begin{corollary}\label{fiberedslice}
If $K$ is a fibered, hyperbolic slice knot and $r$ is a reducing slope, then $r \leq g$.
\end{corollary}

Recall that, for slice knots $K$, $S^3_{p/q}(K)$ is homology cobordant to $L(p,q)$. Then Theorem \ref{thm:redslice} immediately implies:

\begin{corollary}
For any slice knot $K$ with a reducing slope $r$ there is a homology sphere $Y$ and a homology cobordism from $L(r,1)$ to $L(r,1) \# Y$.

\end{corollary}
The proof of Theorem \ref{thm:redslice} also gives the following slice genus bounds.

\begin{corollary}{\label{slicegenus}}
Suppose $K$ admits a reducible surgery of the form $S^3_{pq}(K)\cong Y_p\#Y_q$ with $p>q > 1$ and relatively prime. Then $g_s(K)\geq \frac{(p-1)(q-1)}{2}>0$.
\end{corollary}
\noindent\textbf{Remark}: Corollary \ref{slicegenus} implies that if a slice genus $1$ knot $K$ has a reducible surgery with two summands carrying non-trivial homology, then the reducing slope is $6$. This implies that the two summands conjecture is true for all reducing slopes on slice genus $1$ knots except for the possibility that $S^3_6(K) \cong L(2,1)\#L(3,2)\#Y$ for $Y$ an irreducible homology sphere with $d(Y)=0$. Similarly the only possible reducing slopes for slice genus $2$ knots that could produce more than two summands are $r=6$ and $r=10$. As far as we know, Heegaard Floer theoretic invariants cannot obstruct three summands from appearing in these surgeries.

Next, we investigate how Theorem \ref{thm:redslice} may be applied to the problem of multiple reducing slopes for slice knots. This theorem and its proof are inspired by \cite{HLZ15}*{Theorem 1.6}. For the definition of $\nu$, see Definition \ref{nudef}.

\begin{theorem}{\label{multipleslopes}}

Suppose $K$ is a knot in $S^3$ with $\nu(K) <g(K)$ that also admits reducing slopes $r$ and $r+1$. Further, suppose that both $r$ and $r+1$ surgery split off an integer homology sphere summand. Then $r+1\leq g(k)$.

\end{theorem}

\begin{corollary}{\label{slicemultipleslopes}}

Suppose that $r$ and $r+1$ are simultaneous reducing slopes for a slice knot $K$ in $S^3$. Then $r+1\leq g(K)$. 

\end{corollary}




The proof of Theorem \ref{thm:redslice} also enables us to bound the thickness of a hyperbolic, L-space knot under specific assumptions on its associated reducible surgery. The thickness of $K$, denoted $th(K)$, is defined to be the maximal $\delta$-grading among generators of $\widehat{\textit{HFK}}(K)$, where $\delta(x)=A(x)-M(x)$ is the difference of Alexander and Maslov gradings.

\begin{theorem}
Suppose $K$ is an L-space knot of genus $g$, and admits a reducible surgery of the form $S^3_{2g-1}(K) \cong Y_p \# Y_q$ with $p>q>1$ relatively-prime. Then
\[
q-1 \leq th(K) \leq \frac{p+q}{2}.
\]
\label{thm:lspacethickbound}
\end{theorem}
\noindent If $K$ is a hyperbolic, L-space knot with a reducible surgery as above with $p-q=2$, then its Alexander polynomial is completely determined. Its explicit form is provided in Subsection \ref{subsec:Alexanderpoly}.

\begin{corollary}
Suppose $K$ is an L-space knot of genus $g$, and admits a reducible surgery of the form $S^3_{2g-1}(K) \cong Y_p \# Y_q$ with $p,q>1$ relatively-prime and $p-q=2$. Then $K$ is unique up to knot Floer homology.
\end{corollary}

\noindent Theorem \ref{thm:lspacethickbound} together with Theorems \ref{fiberedHprime} and \cite[Theorem 1.3]{HLZ15} provide some obstruction to reducible surgery along a hyperbolic, L-space knot.

\begin{corollary}
Suppose $K$ is a hyperbolic, L-space knot and that $t$ is the smallest prime factor of $2g(K)-1$. If $th(K) < t-1$, then $K$ does not admit a reducing slope.
\label{cor:smallthick}
\end{corollary}

We say $K$ is thin if all generators of $\widehat{\textit{HFK}}(K)$ lie in the same $\delta$-grading, or equivalently $th(K)=0$. Using \cite[Theorem 1.2]{DeY21b} and Corollary \ref{cor:smallthick}, we have the following.

\begin{theorem}
Thin knots satisfy the Cabling Conjecture.
\label{thm:thinCC}
\end{theorem}





\section*{Organization}
We only consider surgeries with positive slopes, and mirror knots to achieve this whenever necessary. All manifolds are assumed to be compact, connected, oriented 3-manifolds, unless otherwise stated. Coefficients in Floer homology are also taken to belong to $\mathbb{F}=\mathbb{F}_2$.

\section*{Acknowledgements}
The authors are grateful to Tye Lidman for suggesting that \cite{Acuna}*{Proposition 1.4} may be used for some cases involving L-space knots. H.B. thanks his advisor Robert Lipshitz for helpful conversations and support.

\section{Background}

In this section we review the background necessary to prove Theorem \ref{thm:redslice}.

\subsection{$\Spinc$ Structures}

Let $\Spinc(Y)$ denote the set of $\Spinc$ structures on $Y$, and recall that $\Spinc(Y)$ is an affine copy of $H^2(Y;\Z)$. Given a choice of $\Spinc$ structure $s_0$ on $Y$, every other $\Spinc$ structure satisfies $s=s_0+a$ for some $a \in H^2(Y;\Z)$. Furthermore we have an identification $\Spinc(Y_1\#Y_2)=\Spinc(Y_1)\times \Spinc(Y_2)$ and the projection maps onto each factor, $\pi_{Y_1}$ and $\pi_{Y_2}$, intertwine the conjugation actions. Therefore, for $s \in \Spinc(Y_1\#Y_2)$ a spin structure, both $\pi_{Y_1}(s)$ and $\pi_{Y_2}(s)$ are spin structures on $Y_1$ and $Y_2$ respectively.
Next, observe that if $p=|H_1(Y_1;\Z)|$, then $\pi_{Y_1}(s+p)=\pi_{Y_1}(s)$ for any $s \in \Spinc(Y_1\#Y_2)$. This gives a relation among the $d$-invariants of reducible three-manifolds that arise as Dehn surgery along a knot in $S^3$. 


For surgeries on knots in $S^3$, we fix throughout an identification of $\Spinc(S^3_{p/q}(K))$ with $\Z/p\Z$, given by $\sigma : \Z/p\Z \to \Spinc(S^3_{p/q}(K))$ which sends $[i] \to \sigma([i])$ and satisfies $\sigma([i+1])-\sigma([i])=[K'] \in H_1(S^3_{p/q}(K)) \cong \Spinc(S^3_{p/q}(K))$, where $[K']$ is the homology class of the dual knot. For more details on this assignment, see \cite{doflens}*{Section4.1}. We will often abuse notation and write $i$ or $[i]$ for the image of $[i]$ under the map $\sigma$.


\subsection{Heegaard Floer Homology}

Heegaard Floer homology is an invariant of closed, oriented three manifolds that was introduced by Oszv\'{a}th and Szab\'{o} in \cite{Holomorphic}. We will assume familiarity with all flavors of Heegaard Floer homology, as well as the $\Z \oplus \Z$-filtered knot Floer complex $\textit{CFK}^{\infty}(K)$ for knots $K$ in $S^3$ defined in \cite{Holomorphic} and \cite{Ras}. For the readers convenience, we give a brief review of the structure of $\textit{HF}^+(Y,\mathfrak{s})$, the properties of the $d$ invariants, and the mapping cone formula, since they will be used in our main arguments in the next section.

Given a rational homology three-sphere, consider the invariants $\widehat{\textit{HF}}(Y)$, $\textit{HF}^+(Y)$, and $\textit{HF}^{\infty}(Y)$. These are a finite dimensional $\F$-vector space, an $\F[U]$-module, and an $\F[U,U^{-1}]$-module respectively. Further, we have $\textit{HF}^{\circ}(Y) \cong \bigoplus_{\mathfrak{s} \in \Spinc(Y)} HF^{\circ}(Y,\mathfrak{s})$ for $\circ \in \{\,\widehat{\,\,} \,,+,\infty\}$. 

For any rational homology three-sphere $Y$ with $\mathrm{Spin^c}$ structure $\mathfrak{s}$, we have $HF^{\infty}(Y,\mathfrak{s}) \cong \F[U,U^{-1}]$. Also, $HF^+(Y,\mathfrak{s})$ decomposes non-canonically into two pieces. The first is the image of $HF^{\infty}(Y,\mathfrak{s})$ in $HF^+(Y,\mathfrak{s})$. This summand is isomorphic to $\F[U,U^{-1}]/U\F[U]$, which is called the \textit{tower} and is denoted $\mathcal{T}^+$. The grading of $1 \in \mathcal{T}^+$ is an invariant of the pair $(Y,\mathfrak{s})$ and is denoted $d(Y,\mathfrak{s})$. The rational number $d(Y,\mathfrak{s})$ is called the correction term or $d$ invariant. The second summand in $HF^+(Y,\mathfrak{s})$ is the quotient by the image of $HF^{\infty}(Y,\mathfrak{s})$ and is denoted $HF_{red}(Y,\mathfrak{s})$. It is a finite dimensional $\F$ vector space annihilated by a high enough power of $U$.

The $d$-invariant has many useful properties. Among them we mostly use the following:
\begin{enumerate}
\item Suppose $\overline{\mathfrak{s}}$ is the image of $\mathfrak{s}$ under conjugation. Then $d(Y,\overline{\mathfrak{s}})=d(Y,\mathfrak{s})$.

\item For pairs $(Y_1,\mathfrak{s}_1)$ and $(Y_2,\mathfrak{s}_2)$, 
\begin{equation}\label{eqn: additivity}
d(Y_1\#Y_2,\mathfrak{s}_1\# \mathfrak{s}_2)=d(Y_1,\mathfrak{s}_1)+d(Y_2,\mathfrak{s}_2)
\end{equation}
\item $d$ is a homology cobordism invariant: If $W: Y \to Y'$ is a $\Z$ homology cobordism and there is a $\Spinc$ structure on $W$ that restricts to $\mathfrak{s}$ on $Y$ and $\mathfrak{s'}$ on $Y'$, then $d(Y,\mathfrak{s})=d(Y',\mathfrak{s'})$. 

\end{enumerate}

Due to \cite{OS08}, the $d$-invariants of $(p/q)$-surgery along a knot are related to the $d$-invariants of $L(p,q)$. The latter are determined in \cite{doflens}*{Proposition 4.8}, where they show 
\begin{equation}\label{eqn: dinvts}d(L(p,q),i)=\dfrac{(2i+1-p-q)^2-pq}{4pq}-d(L(q,r),j),\end{equation}
with $r$ and $j$ the reductions of $p$ and $i$ modulo $q$, respectively. This formula together $d(S^3)=0$ allows one to determine the $d$-invariants of a lens space recursively. 

Due to \cite{LipshitzLee}*{Proposition 5.3}, the $d$-invariants of $L(p,a)$ also satisfy the quite useful relation

\begin{equation}
d(L(p,a),[i])-d(L(p,a),[a+i])=\frac{p-2i-1}{p}.
\label{eqn:LLdinvt}
\end{equation}

\subsection{The Mapping Cone Formula and the $\nu^+$ Invariant}\label{mappingcone}

In this section, we establish some terminology and notation for the mapping cone formula and the $\nu^+$ invariant. For more details, see \cite{fourgenus} and \cite{mappingcone}. Material from this section will be used to establish the claimed slice genus bounds and the bound on multiple reducing slopes. 

As above, we write $HF^{\circ}$ to mean either the plus or hat version of Floer homology. Let $C=CFK^{\infty}(S^3,K)$ denote the knot Floer complex associated to $K$. This is a $\Z \oplus \Z$ filtered $\Z$ graded chain complex over $\F[U,U^{-1}]$, where the $U$ action lowers the filtration degree by one. Associated to $C$ are the following quotient and sub-quotient complexes useful for computing the plus and hat versions of Floer homology of manifolds arising as Dehn surgery along $K$. To this end, define:
$$A^+_k:=C\{\max\{i,j-k\} \geq 0\} \quad \mathrm{and} \quad \widehat{A}_k:=C\{\max\{i,j-k\}=0\}$$ as well as $$B^+:=C\{i \geq 0\} \quad \mathrm{and} \quad\widehat{B}:=C\{i=0\}$$ where $i$ and $j$ refer to the two filtration degrees. From the definition of $CFK^{\infty}(S^3,K)$ the complex $B^{\circ}$ is isomorphic to $CF^{\circ}(S^3).$ 

There is an obvious map $v^+_k:A^+_k \to B^+$ defined by projection. Similarly, there is a map $h^+_k:A^+_k \to B^+$ which projects to $C\{j \geq k\}$, shifts to $C\{ j \geq 0\}$ via multiplication by $U^k$, and then applies a chain homotopy equivalence between $C\{i \geq 0\}$ and $C\{j \geq 0 \}$ (both of which compute $CF^+(S^3)$, so by general theory are chain homotopic). There are similar maps for the hat versions. 

Just as $HF^+(Y,\mathfrak{s})$ decomposes as a tower and a reduced part, the homology of the quotient complexes $A^+_k(K)$ decompose non-canonically as $\mathcal{T}^+ \oplus A^{red}_k$. The maps $v^+_k$ and $h^+_k$ are isomorphisms for large values of $k$ and so represent multiplication by some non-negative power of $U$, say $U^{V_k}$ and $U^{H_k}$ respectively when restricted to the tower summand in each $A^+_k$. The non-negative integers $V_k$ and $H_k$ are concordance invariants of $K$, satisfy $H_k=V_{-k}$, $H_k=V_k+k$, and 
\begin{equation}
V_k-1\leq V_{k+1}\leq V_k.
\end{equation}
Furthermore, for each $i$, we have \cite{Ras}*{Corollary 7.4}
\begin{equation}\label{localh}
V_i \leq \left\lceil\frac{g_4(K)-i}{2}\right\rceil.
\end{equation}
The $V_i$ also determine the correction terms or $d$ invariants of surgery along the knot $K$:

\begin{theorem}\cite{cosmetic}*{Proposition 1.6}\label{Ni-Wu} For $p,q \geq 0$ and $0 \leq i \leq p-1$. we have:
\begin{equation}\label{eqn: Ni-Wu} d(S^3_{p/q}(K),i)=d(L(p,q),i)-2\max\{V_{\lfloor i/q \rfloor},V_{\lfloor \frac{p+q-1-i}{q} \rfloor}\}
\end{equation}

\end{theorem}

Now we explain how the maps $v_k$ and $h_k$ together with the quotient complexes $A^+_k$ determine the Heegaard Floer homology of $p/q$ surgery along $K$. Since we are only interested in integer surgery in this paper, we write the theorem down in this case. The reader interested in the change to the case of fractional surgeries and a more detailed explanation of the notation should consult \cites{nu,mappingcone}. To this end, let $$\mathcal{A}^{\circ}_{i,p}(K):=\bigoplus_{n \in \Z} (n,A^{\circ}_{i+pn}), \hspace{.2in} \mathcal{B}^{\circ}:=\bigoplus_{n \in \Z} (n,B^{\circ}).$$ Then define a chain map $D_{i,p}^{\circ}: \mathcal{A}^{\circ}_{i,p} \to \mathcal{B}^{\circ}$ by $D_{i,p}^{\circ}(\{(k,a_k)\}_{k \in \Z})=\{(k,b_k)\}_{k \in \Z}$ where $b_k=v^{\circ}_{i+pk}(a_k)+h^{\circ}_{i+p(k-1)}(a_{k-1})$. Letting $\mathbb{X}^{\circ}_{i,p}$ denote the mapping cone of $D^{\circ}_{i,p}$, we have

\begin{theorem}\cite{nu}*{Theorem 1.1}
There is a relatively graded isomorphism of $\F[U]$-modules 

$$H_*(\mathbb{X}^{\circ}_{i,p}) \cong HF^{\circ}(S^3_p(K),i).$$

\end{theorem}

Next, we introduce the $\nu^+$ invariant. as defined in \cite{fourgenus}*{Definition 2.1}.

\begin{definition} The invariant $\nu^+$ is defined as follows: $\nu^+:= \min \{k \in \Z \mid v_k:A^+_k\to \widehat{CF}(S^3), v_k^+(1)=1\},$
where $1 \in H_*(A^+_k)$ is a generator with lowest grading of the tower summand.
\end{definition}

\noindent Recall that $\nu^+(K) \leq g_s(K)$ \cite{fourgenus}*{Proposition 2.4}. With the mapping cone formula we can give an alternative definition of $\nu^+$. This definition is equivalen to the one just given since the integers $V_k$ determine the map $v_k^+$ on the non-torsion summand of $A_k^+$ \cite{cosmetic}.


\begin{lemma}{\label{nu+}} $\nu^+(K)=\min\{k \in \Z \mid V_k=0\}.$

\end{lemma}
\noindent We will also make use of the hat version $\nu$ as defined in \cite{nu}*{Definition 9.1}

\begin{definition}\label{nudef}
For a knot $K \subset S^3$, define $\nu(K):= \min\{ s \mid (\hat{v}_s)_* \neq0\}$. 
\end{definition}
\noindent Then genus detection of knot Floer homology implies that $$g(K)=\max\{ \nu(K),\{s \mid \dim \widehat{A}_{s-1} >1\}\}.$$

\section{Reducible Surgeries on Slice Knots}

In this section we prove Theorems \ref{thm:redslice} and \ref{multipleslopes}.

\subsection{The $d$-invariants of Reducible Manifolds}

Theorem \ref{thm:redslice} follows from the more general Theorem \ref{thm:sumdinvariants} below, which deals with $d$-invariants of knots which admit a reducible surgery. 

\begin{theorem}\label{V_i's}
Suppose $K$ is a knot in $S^3$ such that $pq$ is a reducing slope with $(p,q)=1$, and $S^3_{pq}(K) \cong Y_p \# Y_q$, where $Y_p$ and $Y_q$ are 3-manifolds with $H^2(Y_p;\Z)\cong \Z/p\Z$ and $H^2(Y_q;\Z)\cong \Z/q\Z$. Then for each $0 \leq \ell \leq \frac{(p-1)(q-1)}{2}$, the $V_i$'s satisfy
\begin{equation}
\sum_{i=0}^{q-1}\left(V_{\ell+i}-V_{\alpha(\ell+i+p)}\right)=\frac{(p-1)(q-1)}{2}-\ell,
\label{eq:sumdinvariants}
\end{equation}
Where $\alpha(j)=\min\{j,pq-j\}$.
\label{thm:sumdinvariants}
\end{theorem}

\begin{proof}
Suppose $K$ is a knot in $S^3$ with $S^3_{pq}(K) \cong Y_p \# Y_q$. We will write $\pi_p$ for the projection map $\pi_{Y_p}: \Spinc(S^3_{pq}(K)) \to \Spinc(Y_p)$ and similarly for $\pi_q$. As $|H^2(Y_p;\Z)|=p$ we have $\pi_{p}([p+i])=\pi_{p}([i])$ for $[i], [p+i] \in \Spinc(S^3_{pq}(K))$. Then by additivity of the $d$-invariants, for any $\ell \in \mathbb{Z}$ we have:
\begin{equation}\label{eqn: difference}
d(S^3_{pq}(K),[p+i+\ell])-d(S^3_{pq}(K),[i+\ell])=d(Y_q,\pi_q[p+i+\ell])-d(Y_q,\pi_q[i+\ell]).\end{equation}
By Theorem \ref{Ni-Wu} and Equation \ref{eqn: dinvts}, we see that the left hand side difference equals $$\frac{2(\ell+i)+p(1-q)}{q}+2V_{\ell+i}-2V_{\alpha(\ell+i+p)}.$$ Summing from $i=0$ to $i=q-1$ we see:
\begin{align}
\sum_{i=0}^{q-1} d(S^3_{pq}(K),[\ell+i+p])-d(S^3_{pq}(K),[\ell+i]) &=
\sum_{i=0}^{q-1}\left(\frac{2(\ell+i)+p(1-q)}{q}\right) \nonumber\\
             & \,\,\,\,\,\,\,\,\, +2\sum_{i=0}^{q-1} \left(V_{\ell+i}-V_{\alpha(\ell+i+p)}\right).
\label{eqn:turbosum}
\end{align}
On the other hand, the left-hand side of Equation \ref{eqn:turbosum} is equal to
$$\sum_{i=0}^{q-1} d(Y_q,\pi_q[\ell+i+p])-d(Y_q,\pi_q[\ell+i]).$$
by equation \ref{eqn: difference}. Additionally, this sum is zero because the projection $\pi_q$ induces bijections between $\Spinc(Y_q)$ and both sets $\{\ell,\dots,\ell+q-1\}$, $\{p+\ell,\dots,p+\ell+q-1\}$. Rearranging the sum in Equation \ref{eqn:turbosum}, we see that
$$\sum_{i=0}^{q-1}\left(V_{\ell+i}-V_{\alpha(\ell+i+p)}\right)=\frac{(p-1)(q-1)}{2}-\ell.$$
\end{proof}

\begin{proof}[Proof of Theorem \ref{thm:redslice}]

We first show that Theorem \ref{V_i's} implies $Y$ is an integer homology sphere. Since $V_i(K)=0$ for all $i$, Equation \ref{eq:sumdinvariants} with $\ell=0$ implies $(p-1)(q-1)=0$, so either $p=1$ or $q=1$. If $p=1$, the positive solution of the two summands conjecture in the case where $V_i(K)=0$ for all $i\geq 0$ (\cite{threesummands}) implies that $q$ was not a reducing slope since $Y$ is irreducible. Therefore, under the assumption that we have a reducing slope it follows that $q=1$ and the reducible surgery is $S^3_{p}(K) \cong L(p,a) \# Y$ where $Y$ is an irreducible homology sphere.

To finish off the proof of Theorem \ref{thm:redslice}, it remains to show that $a=1$ and $d(Y)=0$. Using Equation \ref{eqn:LLdinvt}, we have

$$d(L(p,a),[i])-d(L(p,a),[a+i])=\frac{p-2i-1}{p}.$$ There are two cases to consider since $\pi_L$ preserves self-conjugacy of $\text{Spin}^{c}$ structures, and so either $\pi_L([0])$ is $\frac{a-1}{p}$ or $\frac{p+a-1}{p}$. First, suppose that $\pi_L([0])=\frac{a-1}{p}$ and $[s] \in \Spinc(S^3_p(K))$ satisfies $\pi_L([s])=\pi_L([0])+a$. Then

\begin{align*}
\frac{p-a}{p}=d(L(p,1),[0])-d(L(p,1),[s])-2(V_0-V_{\alpha(s+p)})=\frac{-s^2}{p}+s.
\end{align*}
This implies that $s(p-s)=p-a$. However, $s(p-s)\geq p-1$, and so we must have $a=1$.

Next, suppose that $\pi_L[0]=\frac{p+a-1}{2}$ and $\pi_L([s])=\pi_L([0])+a$. The same computation as above gives

\begin{align*}
\frac{-a}{p}=d(L(p,a),0)-d(L(p,a),[s])-2(V_0-V_{\alpha(s+p)})=\frac{-s^2}{p}+s
\end{align*}

\noindent This implies that $s(p-s)=-a$ a contradiction. The fact that $d(Y)=0$ follows immediately from additivity of the $d$-invariants.

\end{proof}

\begin{proof}[Proof of Corollary \ref{slicegenus}] Suppose $K$ is a knot in $S^3$ which admits a reducing slope of the form $r=pq$ with $S^3_{pq}(K) \cong Y_p\#Y_q$. Then equation \ref{eq:sumdinvariants} implies that $V_i \neq 0$ for $i < \frac{(p-1)(q-1)}{2}$. Therefore, $\nu^+(K) \geq \frac{(p-1)(q-1)}{2}$ by Lemma \ref{nu+}. Since $\nu^+$ is a lower bound for the slice genus of a knot, the result follows.

\end{proof}

\subsection{Multiple Reducing Slopes on Slice Knots}


In this section we use the mapping cone formula for $\HFhat$ and the following lemma to prove Theorem \ref{multipleslopes}. The lemma below follows immediately from the Kunneth theorem for $\HFhat$ and the fact that lens spaces are $L$-spaces (for the proof and for the analogous statement for $HF^+$, see \cite{HLZ15}*{Lemma 2.6}.)

\begin{lemma}{\label{sym}}

Suppose $Y$ is a three manifold and $Y \cong L(p,a) \# Y_2$ with $|H^2(Y_2,\Z)|=q$. Then for any $\alpha \in H^2(Y)$ and $\mathfrak{s} \in \Spinc(Y)$ we have $\dim\HFhat(Y,\mathfrak{s}+q\alpha)=\dim\HFhat(Y,\mathfrak{s})$.

\end{lemma}
\begin{proof}[Proof of Theorem \ref{multipleslopes}]

Suppose $S^3_r(K) \cong L(r,a) \#Y$ and $S^3_{r+1}(K) \cong L(r+1,b)\#Z$ where $Y$ and $Z$ are both integer homology spheres, $(r,a)=1$ and $(r+1,b)=1$, and $r \geq g(K)$. We assume the two reducing slopes are consecutive positive integers $r$ and $r+1$ by mirroring the knot if necessary. Since both surgeries split off integer homology three spheres, we see by Lemma \ref{sym} that

$$\dim(\HFhat(S^3_{r}(K),i))=\dim(\HFhat(S^3_{r}(K),j))$$
for any two $\Spinc$ structures $i$ and $j$ on $S^3_r(K)$. Similarly,

$$\dim(\HFhat(S^3_{r+1}(K),i))=\dim(\HFhat(S^3_{r+1}(K),j)).$$

Note that if $r$ is odd, we choose representatives $[i]$ of $\Spinc$ satisfying $-\lfloor r/2 \rfloor \leq i \leq \lfloor r/2\rfloor$. If $r$ is even, choose representatives with $-\lfloor r/2 \rfloor< i \leq \lfloor r/2 \rfloor.$
By the assumption $r \geq g(K)$ we have $r=g+i$ for $0 \leq i \leq g-1$. In this case, the mapping cone formula implies that $\HFhat(S^3_r(K),k) \cong H_*(\widehat{A}_k)$ for $k=-i,-i+1,\dots,0,\dots,i-1,i$. For all other $k$ between $-\lfloor r/2\rfloor$ and $\lfloor r/2 \rfloor$ we have

$$\HFhat(S^3_r(K),k) \cong H_*(\mathrm{Cone}(\widehat{A}_{k-i-g} \oplus \widehat{A}_{k} \to \F)).$$ 
Now, consider the mapping cone for $r+1$ surgery. Since $r+1=g+i+1$, $\HFhat(S^3_{r+1}(K),k) \cong H_*(\widehat{A}_k)$ for $k=-i-1,-i,-i+1,\cdots,i-1,i,i+1$ and for all other $k$ between $-\lfloor (r+1)/2\rfloor$ and $\lfloor (r+1)/2\rfloor$ we have

$$\HFhat(S^3_{r+1}(K),k) \cong H_*(\Cone(\widehat{A}_{k-i-1-g} \oplus \widehat{A}_k\to \F)).$$

Let $n_k=\dim(H_*(\widehat{A}_k))$. Then in summary, from $r$ surgery reducing we have
$$n_0=n_1=n_2=\cdots =n_i=n_{i+1}+n_{1-g}+1-2\rank(h_{1-g} \oplus v_{i+1}).$$
However, $r+1$ surgery reducible implies that $n_0=n_{i+1}$, and so

\begin{equation}\label{eqn: multipleslopes}
n_{i+1}=n_{i+1}+n_{1-g}+1-2\rank(h_{1-g} \oplus v_{i+1}).\end{equation}
Since $\rank(h_{1-g}\oplus v_{i+1})=0$ or $1$, we either have $n_{1-g}=-1$ or $n_{1-g}=1$ by Equation \ref{eqn: multipleslopes}. The former case is impossible, so $\rank(h_{1-g}\oplus v_{i+1})=1$ and $H_*(A_{g-1})$ is one dimensional. This contradicts the genus detection of knot Floer homology. In particular, recall from \cite{nu} that the map $\hat{v}_{g-1}$ is not an isomorphism. So either $\nu(K)=g$ or $\dim(H_*(A_{g-1}))>1$. However, we found that $n_{g-1}=1$ and the assumption on $\nu(K)$ implies that the map $(v_{g-1})_*$ is surjective, a contradiction. 
\end{proof}


\begin{proof}[Proof of Corollary \ref{slicemultipleslopes}]
This follows immediately from Theorem \ref{thm:redslice} and Theorem \ref{multipleslopes}.

\end{proof}

\section{Reducible Surgeries on L-space Knots}

In this section, we will show that if an $L$-space knot $K$ admits a reducible surgery, then its thickness is bounded in terms of the reducing slope parameters.

\subsection{Thickness Bounds}

We first show that such a reducible surgery cannot admit integer homology sphere summands, and then collect Lemmas useful for the upcoming subsections.

\begin{proposition}
Suppose $K$ is an L-space knot of genus $g$. If $K$ admits a reducible surgery of the form $S^3_{2g-1}(K) \cong Y_p \# Y_q$ with $p>q$ relatively-prime, then $p,q > 1$.
\label{prop:GAprop}
\end{proposition}

\begin{proof}
Recall that L-space knots are fibered due to \cites{Ghi08, Ni07}. We may use Theorem \ref{fiberedHprime} to see that the lens space summand of $S^3_{2g-1}(K)$ cannot be $L(2g-1,a)$, as otherwise $S^3_{2g-1}(K)$ would fail to be reducible. Thus, we must have $p,q>1$.
\end{proof}

When $p,q>1$, the techniques in the proof of Theorem \ref{thm:sumdinvariants} are available to relate multiple differences of $d$-invariants. We divide the proof of Theorem \ref{thm:lspacethickbound} into two lemmas, since only the lower bound on knot thickness is necessary to prove Theorem \ref{thm:thinCC}.

\begin{lemma}
Suppose $K$ is an L-space knot of genus $g$, also admitting a reducible surgery of the form $S^3_{2g-1}(K) \cong Y_p \# Y_q$ with $p>q>1$ relatively prime, and let $m = \tfrac{p-q}{2}$.
\begin{itemize}
\item If $m=1$, then $q-1 \leq th(K)$.
\item If $1<m<q+2$, then $q+m-2 \leq th(K)$.
\item If $m \geq q+2$, then $2(q-1) \leq th(K)$.
\end{itemize}
\label{lem:lspacethicklbound}
\end{lemma}

\begin{proof} We aim to use Theorem \ref{thm:sumdinvariants} to show that $K$ admits a number of consecutive $V_i$'s that are equal, and involve grading information by relating this to $\textit{CFK}^{\infty}(K)$. To see this, recall that the $V_i$'s determine $\textit{CFK}^{\infty}(K)$ when $K$ is an L-space knot. Having $n$ consecutively equal $V_i$'s manifests in the full knot Floer complex as a pair of generators $a$ and $b$ with a length $n-1$ vertical differential between them. Then we have
\[
A(a)-A(b) = n-1 \,\,\,\, \text{and} \,\,\,\, M(a)-M(b) = 2(n-1)-1.
\]
Thus $\delta(b)-\delta(a) = n-2$, showing that $th(K) \geq n-2$.

To search for a string of consecutively equal $V_i$'s, let $m=\tfrac{p-q}{2}$ and $k = \tfrac{(p-1)(q-1)}{2}-1 = \tfrac{pq-p-q-1}{2}$. Since $2g-1=pq$, notice that $k = g-(q+m+1)$. Using Equation \ref{eq:sumdinvariants} with this choice of $k$ in place of $\ell$ yields
\begin{equation}
\sum_{i=0}^{q-1} (V_{k+i}-V_{\alpha(k+i+p)}) = 1.
\label{eqn:lspacethick1}
\end{equation}
To better understand the various $\alpha(k+i+p)$, observe that for $0 \leq i \leq q-1$ we have both
\begin{align*}
k+i &\leq k+q-1 \\
&= g-m-2\\
&< g-1,
\end{align*}
and also
\begin{align*}
k+i+p &\geq k+p \\
&\geq g-(q+m+1)+p\\
&\geq g+m-1\\
&> g.
\end{align*}
Together these imply that Equation \ref{eqn:lspacethick1} becomes
\begin{equation}
\sum_{i=0}^{q-1} \left(V_{k+i}-V_{g-m-i} \right) = 1,
\label{eq:sumtermcancel}
\end{equation}
where $\alpha(k+i+p)=g-(m+i)$. We claim that all but 2 terms of this sum cancel. To see this, notice that $j=\tfrac{q+1}{2}$ satisfies $k+j=\alpha(k+j+p)$. All terms after $(V_{k+j}-V_{\alpha(k+j+p}))$ in Equation \ref{eq:sumtermcancel} then cancel with their corresponding mirror earlier in the sum. This leaves just the first $q-(1+2((q-1)-(\tfrac{q+1}{2}))) = 2$ terms remaining. Thus, 
\[
(V_k-V_{g-m})+(V_{k+1}-V_{g-(m+1)}) = 1.
\]
We must have $V_{k+1} = V_{g-(m+1)}=V_{k+q}$ since the second difference is nested within the first, and the $V_i's$ are non-decreasing. This yields $q$ consecutively equal $V_i$'s. 

Let us call a string of $n$ consecutively equal $V_i$'s an $n$-\textit{block}. Additionally, we will suggestively call the $q$-block containing the $V_i$'s between $V_{k+1}$ and $V_{k+q}$ the \textit{central} $q$-block. Determining which types of $n$-blocks may appear above and below the central $q$-block will allow us to find strings of consecutively equal $V_i$'s for the cases depending on $m$. For example, observe that $k+q = g-2$ if $m=1$. Both $V_{g-1}=1$ and $V_{g-2}=1$ since $K$ is an L-space knot \cite{heddengeography}*{Corollary 9}, and so $V_{k+1}=V_{g-1}$ provides a string of $q+1$ consecutively equal $V_i$'s.

Suppose $m \geq q+2$. Consider the following equation obtained from Equation \ref{eq:sumdinvariants} using $\ell=k-q$,
\begin{equation}
\sum_{i=0}^{q-1}(V_{k-q+i}-V_{\alpha(k-q+i+p)})=q.
\label{eq:underqblock}
\end{equation}
We see that $\alpha(k-q+1+p) = k+2q$ when $m \geq q+2$, and so Equation \ref{eq:underqblock} implies that $(V_{k-q+1}-V_{k+2q}) = 1$. The other $q-1$ differences are nested within this one, and so we must have $q$-blocks above and below the central one. Precisely two of these three must have their contained $V_i$'s agree, and so $K$ admits $2q$ consecutively equal $V_i$'s.

Otherwise we have $1 < m < q+2$. Here we see that $j = q+1-m$ satisfies that $(V_{k-q+j}-V_{g-1})$ is a term of Equation \ref{eq:underqblock}. This implies that $V_{k-m+1} = V_{k-q+j} = V_{g-1}+1 = 2$. Either $1=V_{g-1} = V_{k+1}$ or $2=V_{k+q} = V_{k-m+1}$. In both cases, $K$ admits $q+m$ consecutively equal $V_i$'s.
\end{proof}

Before turning to an upper bound on thickness, the recently obtained lower bound enables us to prove Corollary \ref{cor:smallthick}.

\textit{Proof of Corollary \ref{cor:smallthick}}
As an L-space knot, the only reducing slope to consider is $2g(K)-1$ due to \cite{HLZ15}. Suppose $K$ admits a reducible surgery of the form $S^3_{2g-1}(K) \cong Y_p \# Y_q$, with $p>q$ relatively prime. Proposition \ref{prop:GAprop} forces $q > 1$, but Lemma \ref{lem:lspacethicklbound} implies $t-1 \leq q-1 \leq th(K)$, which is the desired contradiction.
\qed

The upper bound that we can determine is comparatively not as sharp.

\begin{lemma}
Suppose $K$ is an L-space knot of genus $g$, also admitting a reducible surgery of the form $S^3_{2g-1}(K) \cong Y_p \# Y_q$ with $p>q>1$ relatively prime, and let $m = \tfrac{p-q}{2}$. Then
\[
th(K) \leq q+m.
\]
\label{lem:lspacethickubound}
\end{lemma}

\begin{proof}
As before, choose $k = \tfrac{(p-1)(q-1)}{2}-1$ so that Equation \ref{eq:sumdinvariants} yields
\[
\sum_{i=0}^{q-1}(V_{k+i}-V_{\alpha(k+i+p)})=1.
\]
Due to Lemma \ref{lem:lspacethicklbound}, we see that $K$ admits a central $q$-block of $q$ consecutively equal $V_i$'s between $V_{k+1}=V_{k+q}$. The nesting behavior of the differences of $d$-invariants observed in the proof of Lemma \ref{lem:lspacethicklbound} ends at $(V_{g-p}-V_{g-1})$, and beyond this point we show that only small strings of consecutively equal $V_i$'s can occur.

Suppose for the sake of contradiction that $j < j+q+1 < g-p$ satisfies $V_j=V_{j+q+1}$, providing a string of $q+1$ consecutively equal $V_i$'s beneath $V_{g-p}$. Notice that $V_{g-p}$ is the first term for which $\alpha(g-p+p) \neq g$, but instead returns $g-1$. This means that any $V_i$ with $i < g-p$ satisfies $\alpha(i) = i+p$. Due to Theorem \ref{thm:sumdinvariants}, there are $(k+1)$-many versions of Equation \ref{eq:sumdinvariants} that each involve a sum of $q$-many differences of $V_i$'s. Because there are $q+1$ equations between those whose sums begin with $(V_j-V_{j+p})$ and $(V_{j+q+1}-V_{j+q+1+p})$, respectively, we must have
\[
(V_{j}-V_{j+p})-(V_{j+q+1}-V_{j+q+1+p}) > 0.
\]
However $V_j=V_{j+q+1}$ would then imply $V_{j+q+1+p}>V_{j+p}$, which is the desired contradiction. Thus, the length of a string of consecutively equal $V_i$'s appearing beneath $V_{g-p}$ is at most $q$. Since we already know that $K$ admits a string of $q+1$ consecutively equal $V_i$'s, we need only investigate above $V_{g-p}$.

Due to our choice of $k$, we see that $V_k$ sits just beneath the central $q$-block, since $V_{k}-V_{\alpha(k)}$ is the only non-zero term of Equation \ref{eqn:lspacethick1} that remains. If $V_k=2$, then $K$ admits a string of $g-1-(k+1)+1 = q+m = \tfrac{p+q}{2}$ consecutively equal $V_i$'s between $V_{g-1}$ and $V_{k+1}$. Alternatively, $K$ can have a few more consecutively equal $V_i$'s. To see this, let $B_q$ denote the number of \textit{full} $q$-blocks between the central $q$-block and $V_{g-2}$. Notice that $V_{g-p-1} = 1+B_q$ since the number of full $q$-blocks above and below the central $q$-block is the same. Then if $V_k=1+B_q$ and all of $V_{g-p}$ and $V_{g-p-2}$ belong to the same $q$-block, then $K$ would admit at most $k-(g-p-2)+1 = \tfrac{p+q}{2}+2$ many consecutively equal $V_i$'s.
\end{proof}

\subsection{Alexander polynomials}
\label{subsec:Alexanderpoly}

\indent Back in the proof of Lemma \ref{lem:lspacethicklbound}, we connected knowledge of the $V_i$'s to the form of $\textit{CFK}^{\infty}(K)$. A quadruple $V_{i+2}, V_{i+1}, V_{i}, V_{i-1}$ satisfying $V_{i-1}=V_i=V_{i+1}+1=V_{i+2}+1$ corresponds to a pair of generators $x_{i+1}$ and $x_i$ connected by a length 1 horizontal differential from $U^{-1}{x_i}$ to $x_{i+1}$. Since knot Floer homology categorifies the Alexander polynomial, we see that $\Delta_K(t)$ then contains a $t^{i+1}-t^i$ term. It turns out in the $p-q=2$ case that all of the $V_i$'s above $V_k$ are determined. This is enough to determine all of the $V_i$'s beneath $V_k$, and in turn the Alexander polynomial of $K$. The theorem below is stated in terms of only the non-negative powers of the symmetrized Alexander polynomial, denoted $\Delta^+_K(t)$, for simplicity.

\begin{corollary}
Suppose $K$ is an L-space knot of genus $g$, also admitting a reducible surgery of the form $S^3_{2g-1}(K) \cong Y_p \# Y_q$ with $p>q>1$ relatively prime and $p-q=2$. 
If $p \equiv -1 \, ( \text{mod} \, 4)$, then
\[
\Delta^+_K(t) = 1 + (t^g-t^{g-1}) + \sum_{j=1}^{\tfrac{p-3}{4}} \left( t^{2j}-t^{2j-1} \right) + \sum_{i=1}^{\tfrac{q-1}{2}} \sum_{j=1}^{i} \left(t^{g-ip+2j-1}-t^{g-ip+2j-2}\right).
\]
If $p \equiv 1 \, ( \text{mod} \, 4)$, then
\[
\Delta^+_K(t) = (t^g-t^{g-1}) + \sum_{j=1}^{\tfrac{p-1}{4}} \left( t^{2j-1}-t^{2(j-1)} \right) + \sum_{i=1}^{\tfrac{q-1}{2}} \sum_{j=1}^{i} \left(t^{g-ip+2j-1}-t^{g-ip+2j-2}\right).
\]
\label{cor:m1Apoly}
\end{corollary}

\begin{proof}
Our goal is to show that all of the $V_i$'s are determined, and to locate everywhere that the $V_i$'s increase to assemble $\Delta^+_K(t)$. Since $p-q=2$, the central $q$-block satisfies $V_{k+1}=V_{k+q}=V_{g-2}=1$. The rest of the $V_i$'s are determined by sequencing through the $\tfrac{(p-1)(q-1)}{2}$-many versions of Equation \ref{eq:sumdinvariants}. Essentially Equation \ref{eqn:lspacethick1}, repeated here as
\[
\sum_{i=0}^{q-1} \left(V_{k+i}-V_{\alpha(k+i+p)} \right) = 1,
\]
yields $V_k=1+V_{g-1}=2$. Then 
\begin{equation}
\sum_{i=0}^{q-1} \left(V_{(k-1)+i}-V_{\alpha((k-1)+i+p)} \right) = 2
\label{eq:lspacethick2}
\end{equation}
yields $(V_{k-1}-V_{g-1})+(V_{k}-V_{g-1})=2$, which forces $V_{k-1}=2$. In general, first term of 
\begin{equation}
\sum_{i=0}^{q-1} \left(V_{(k-j)+i}-V_{\alpha{(k-j)+i+p}} \right) = j
\label{eq:lspacethickj}
\end{equation}
determines $V_{k-j}$ provided that all $V_i$'s with $i > k-j$ are known.

Next, we locate where the $V_i$'s increase to determine the components of $\Delta^+_K(t)$. The $V_i$'s ranging from $V_{g-1}$ (duplicated once to account for $\alpha(g)=g-1$ ) to $V_0$ can be placed in an array with $p$ rows and $\tfrac{q+1}{2}$ columns. They populate the array first along columns and then by rows, without filling the final column. The first column consists of two copies of $V_{g-1}$ and the central $q$-block, all satisfying $V_{g-1}=V_{k+1}=V_{g-p+1}=1$. The second column contains the $q$-block between $V_{g-p}$ and $V_{g-p-q+1} = V_{g-2p+3}$. The next $q$-many $V_i$'s all have equal differences $(V_i-V_{i+p})=2$ guaranteed by
\[
\sum_{i=0}^{q-1} \left( V_{k-2q+i}-V_{k+p-2q+i} \right) = 2q,
\]
but the corresponding string of $V_{i+p}$'s are not all equal. Since $p-q=2$, the first two $V_i$'s of the next set form a $2$-block at the bottom of the second column, with the remaining $V_i$'s forming a $q-2$-block at the top of the third column. In general, the $i$th column contains a $q-2(i-2)$-block followed by $(i-1)$-many $2$-blocks. Arranged in this way, we can easily determine when the $V_i$'s change. 

Aside from the initial increase from $V_g$ to $V_{g-1}$ contributing $(t^g-t^{g-1})$, the increase from $V_{g-ip+1}$ at bottom of the $i$th column to $V_{g-ip}$ at the top of the $(i+1)$st column contributes $t^{g-ip+1}-t^{g-ip}$. From the second column onward, the $(i-1)$-many $2$-blocks above $V_{g-ip+1}$ also make a contribution. Overall, the contribution from the $i$th column (without $V_{g-(i-1)p}$ but including $V_{g-ip}$) is
\[
\sum_{j=1}^{i} \left(t^{g-ip+2j-1}-t^{g-ip+2j-2} \right).
\]
This procedure works for all but the final $\tfrac{q+1}{2}$ column, where the parity of $\tfrac{p+1}{2}$ matters. Suppose $\tfrac{p+1}{2}$ is even, and notice that $V_0$ belongs to the bottom of the final $2$-block. This is because the final column has $g-(\tfrac{q+1}{2}-1)p=\tfrac{p+1}{2}$-many $V_i$'s, and begins with a $q-2(\tfrac{q+1}{2}-2)=3$-block. Since the $V_i$'s satisfy $V_{-i}-V_i=i$ \cite{HLZ15}*{Lemma 2.5}, we see that $V_1=V_0$ forces an increase after $V_0$. There are $\tfrac{p-3}{4}$-many $2$-blocks, so the contribution from this column is then
\[
1+\sum_{j=1}^{\tfrac{p-3}{4}} \left( t^{2j}-t^{2j-1} \right).
\]

If instead $\tfrac{p+1}{2}$ is odd, then $V_0$ appears beneath the final $2$-block and there are only $\tfrac{p-5}{4}$-many $2$-blocks. The contribution is then
\[
\sum_{j=1}^{\tfrac{p-1}{4}} \left( t^{2j-1}-t^{2(j-1)} \right).
\]



\end{proof}

\bibliographystyle{alpha}
\bibliography{reducibleslicethicknessfinal}

\end{document}